\documentclass[12pt,reqno,twoside]{amsart}
\usepackage{amsmath,amssymb,amsfonts,amsthm,indentfirst,setspace}

\newtheorem{theorem}{Theorem}[section]
\newtheorem{definition}{Definition}[section]
\newtheorem{lemma}{Lemma}[section]
\newtheorem{corollary}{Corollary}[section]
\newtheorem{remark}{Remark}[section]
\newcommand{\bea}{\begin{eqnarray}}
\newcommand{\eea}{\end{eqnarray}}
\newcommand{\eeas}{\end{eqnarray*}}
\newcommand{\beas}{\begin{eqnarray*}}

\begin{document}

\title[Pseudo parallel CR-submanifolds...]{Pseudo parallel CR-submanifolds in a non-flat complex space form}
\thanks{This work was supported in part by the UMRG research grant (Grant No. RG163-11AFR)}

\author{Avik De \and Tee-How Loo}
\address{De, A.\\
Institute of mathematical Sciences, University of Malaya, 50603 Kuala
Lumpur, Malaysia}
\email{de.math@gmail.com}
\address{Loo, T. H. \\
Institute of Mathematical Sciences, University of Malaya, 50603 Kuala Lumpur, Malaysia.}
\email{looth@um.edu.my}

\begin{abstract}
We classify pseudo parallel proper CR-submanifolds in a non-flat complex
space form with CR-dimension greater than one. With this result, the
non-existence of recurrent as well as semi parallel proper CR-submanifolds in a non-flat complex
space form with CR-dimension greater than one can also be obtained.
\end{abstract}

\maketitle

\setcounter{page}{1} \setlength{\unitlength}{1mm}\baselineskip .45cm %
\pagenumbering{arabic} \numberwithin{equation}{section}

\footnotetext{$\mathbf{2000\hspace{5pt}Mathematics\; Subject\;
Classification,\;}:53C40, 53C15$. \newline
}


\section{Introduction}

Let $M$ be an isometrically immersed submanifold in a Riemannian manifold $%
\hat M$. Denote by $\langle,\rangle$ the metric tensor of $\hat M$ as well
as that induced on $M$. Then $M$ is said to be \emph{pseudo parallel} if its
second fundamental form $h$ satisfies the following condition 
\begin{equation*}
\bar R(X,Y)h=f((X\wedge Y)h) 
\end{equation*}
for all vectors $X,Y$ tangent to $M$,  where $f$, called the \emph{%
associated function}, is a smooth function on $M$, $\bar R$ is the curvature
tensor corresponding to the van der Waerden-Bortolotti connection $\bar\nabla$ and 
\begin{align*}
(X\wedge Y)Z&=\langle Y,Z\rangle X-\langle X,Z\rangle Y. 
\end{align*}
In particular, when the associated function $f=0$, $M$ is called a \emph{%
semi parallel} submanifold. It is called \emph{recurrent} if and only if $(\bar\nabla_Xh)(Y,Z)=\tau(X)h(Y,Z),$ where $\tau$ is a 1-form.

Pseudo parallel submanifolds is a generalization of semi parallel and
parallel submanifolds. Parallel submanifolds in a real space form was
completely classified in \cite{ferus}, \cite{takeuchi}. Semi parallel and
pseudo parallel submanifolds in a real space form were also studied
extensively by many researchers (cf. \cite{asperti1}, \cite{asperti2}, \cite%
{deprez}, \cite{dillen}, \cite{lobos}, \cite{lumiste}).

By $n$-dimensional complex space forms $\hat M_n(c)$, we mean complete and
simply connected $n$-dimensional Kaehler manifolds with constant holomorphic
sectional curvature $4c$. For each real number $c$, up to holomorphic
isometries, $\hat M_n(c)$ is a complex projective space $\mathbb{C}P_n$, a
complex Euclidean space $\mathbb{C}_n$ or a complex hyperbolic space $%
\mathbb{C}H_n$ depending on whether $c$ is positive, zero or negative,
respectively.

It is known that a parallel submanifold of a non-flat complex space form $\hat{M}_{n}(c)$, $c\neq 0$,  is either holomorphic or totally real (cf. \cite{chenogiue}). As a result, there does not
exist any parallel real hypersurface in $\hat{M}_{n}(c)$, $c\neq 0$.
Further, the non-existence of semi parallel real hypersurfaces in $\hat{M}%
_{n}(c)$, $c\neq 0$, $n\geq 2$, was proved by Ortega (cf. \cite{ortega}).
Nevertheless, there do exist pseudo parallel real hypersurfaces in $\hat{M}%
_{n}(c)$, $c\neq 0$. Indeed, Lobos and Ortega gave a classification of
pseudo parallel real hypersurfaces in $\hat{M}_{n}(c)$, $c\neq 0$, $n\geq 2$%
, as below: 
\begin{theorem}[\cite{lobos-ortega}]	\label{thm:pseudo-parallel-hypersurface}
Let $M$ be a connected pseudo parallel real hypersurface in $\hat M_n(c)$, $n\geq 2$, $c\neq0$, with associated function $f$. 
Then $f$ is constant and positive, and $M$ is an open part of one of the following real hypersurfaces:

\begin{enumerate}
\item[(a)]
For $c=-1<0$:
	\begin{enumerate}
	\item[(i)] A geodesic hypersphere of radius $r>0$  with $f=\coth^2r$.
	\item[(ii)] A tube of radius $r>0$ over $\mathbb CH_{n-1}$ with $f=\tanh^2r$.
	\item[(iii)] A horoshpere with $f=1$. 
	\end{enumerate}
\item[(b)] For $c=1>0$:
	\begin{enumerate}
	\item[(i)] A geodesic hypersphere of radius $r\in\:]0,\pi/2[$ with $f=\cot^2r$.
	\end{enumerate}
\end{enumerate}
\end{theorem}

Note that a real hypersurface in a Kaehler manifold is a CR-submanifold of
codimension one. A natural problem arisen is to generalize these known
results on real hypersurfaces in $\hat M_n(c)$ into the content of
CR-submanifolds. For technical reasons, certain additional restrictions such
as the semi-flatness assumptions on the normal curvature tensor (cf. \cite%
{yano-kon}), or restriction on the CR-codimension (cf. \cite{djoric}, \cite{loo}), have
been imposed while dealing with CR-submanifolds of higher codimension. It
would be interesting to see if any nice results on CR-submanifolds could be
obtained without these restrictions.

In this paper, we study pseudo parallel proper CR-submanifolds in $\hat{M}%
_{n}(c)$, $c\neq 0$, with none of the above mentioned restrictions. More
precisely, we prove the following: 
\begin{theorem}\label{thm:main}
Let $M$ be a connected proper CR-submanifold in $\hat M_{n}(c)$, $c\neq0$. Suppose that $\dim_{\mathbb C}\mathcal D=p\geq2$.
If $M$ is pseudo parallel with associated function $f$, then $f$ is a positive constant and 
$M$ is an open part of one of the following spaces:
\begin{enumerate}
\item[(a)]
For $c=-1<0$:
	\begin{enumerate}
	\item[(i)] A geodesic hypersphere in $\mathbb CH_{p+1}\subset\mathbb CH_n$ of radius $r>0$  with $f=\coth^2r$.
	\item[(ii)] A tube over $\mathbb CH_{p}$ in $\mathbb CH_{p+1}\subset\mathbb CH_n$ of radius $r>0$ with $f=\tanh^2r$.
	\item[(iii)] A horoshpere in $\mathbb CH_{p+1}\subset\mathbb CH_n$ with $f=1$. 
	\end{enumerate}
\item[(b)] For $c=1>0$:
	\begin{enumerate}
	\item[(i)] A geodesic hypersphere in $\mathbb CP_{p+1}\subset\mathbb CP_n$ of radius $r\in\:]0,\pi/2[$ with $f=\cot^2r$.
	\item[(ii)] An invariant submanifold in a geodesic hypersphere in $\mathbb CP_n$ of radius $r\in\:]0,\pi/2[$ with
				$f=\cot^2r$.
	\end{enumerate}
\end{enumerate}
\end{theorem}

From the above theorem, we see that the associated function $f$ is a
non-zero constant for pseudo parallel proper CR-submanifolds in $\hat M_n(c)$%
, $c\neq0$. Hence we have

\begin{corollary}\label{cor:semi-parallel}
There does not exist any semi parallel proper CR-submanifold $M$ in $\hat M_n(c)$, $c\neq0$, with 
$\dim_{\mathbb C}\mathcal D\geq2$.
\end{corollary}This corollary generalizes the non-existence of semi parallel
real hypersurfaces in $\hat{M}_{n}(c)$, $c\neq 0$ (cf. \cite{ortega}) and
improves a result in \cite{kon}: There does not exist any semi parallel
proper CR-submanifold in $\hat{M}_{n}(c)$, $c\neq 0$, with semi-flat normal
connection.

By applying Corollary~\ref{cor:semi-parallel}, we can then prove the non-existence of proper recurrent CR-submanifolds
in $\hat M_n(c)$, $c\neq0$, with 
$\dim_{\mathbb C}\mathcal D\geq2$ (cf. Corollary~\ref{cor:recurrent}).

The paper is organized as follows:\\
In Section $2$, we fix some notations and
recall some basic material of CR-submanifolds in a Kaehler manifold which we
use later. A fundamental property of Hopf hypersurfaces in $\hat{M}_{n}(c)$, 
$c\neq 0$, is that the principal curvature $\alpha $ corresponding to the
Reeb vector field $\xi $ is constant. Moreover, the other principal
curvatures can be related to $\alpha $ by a nice formula (cf. \cite{ryan}).
We generalize these results to mixed-geodesic CR-submanifolds of maximal
CR-dimension in $\tilde{M}_{n}(c)$ in Section 3. In Section $4$ we
present the proof of Theorem~\ref{thm:main}. In the last section, recurrence and semi-parallelism have been discussed in the context of Riemannian vector bundles. We show that for any homomorphism of Riemannian vector bundles, recurrence directly implies semi-paralellism and thus conclude that there does not exist any proper recurrent CR-submanifold $M$ in $\tilde{M}_n(C)$, $c\neq0$, with $\dim_{\mathbb C}\mathcal D\geq2$
 (cf. Corollary~\ref{cor:recurrent}).


\section{CR-submanifolds in a Kaehler manifold}

\label{sec:cr} 

Let $\hat M$ be a Riemannian manifold, and let $M$ be a connected Riemannian
manifold isometrically immersed in $\hat M$. For a vector bundle $\mathcal{V}
$ over $M$, we denote by $\Gamma(\mathcal{V})$ the $\Omega^0(M)$-module of
cross sections on $\mathcal{V}$, where $\Omega^k(M)$ denotes the space of $k$%
-forms on $M$.

Denote by $\langle ,\rangle $ the Riemannian metric of $\hat{M}$ and $M$ as
well, $h$ the second fundamental form and $A_{\sigma }$ the shape operator
of $M$ with respect to a vector $\sigma $ normal to $M$. Also, let $%
\nabla $ denote the Levi-Civita connection on the tangent bundle $TM$ of $M$
and $\nabla ^{\perp }$, the induced normal connection on the normal bundle $%
TM^{\perp }$ of $M$. The second fundamental form $h$ and the shape operator $%
A_{\sigma }$ of $M$ with respect to $\sigma \in \Gamma (TM^{\perp })$ is
related by the following equation 
\begin{equation*}
\langle h(X,Y),\sigma \rangle =\langle A_{\sigma }X,Y\rangle 
\end{equation*}%
for any $X,Y\in \Gamma (TM)$.

Let $R$ and $R^{\perp }$ be the curvature tensors associated with $\nabla $
and $\nabla ^{\perp }$ respectively. We denote by $\bar{\nabla}$ the van der
Waerden-Bortolotti connection and $\bar{R}$ its corresponding curvature
tensor. Then we have 
\beas
(\bar{R}(X,Y)A)_{\sigma }Z&=&R(X,Y)A_{\sigma }Z-A_{\sigma
}R(X,Y)Z-A_{R^{\perp }(X,Y)\sigma }Z, \\
(\bar{R}(X,Y)h)(Z,W)&=&R^{\perp }(X,Y)h(Z,W)-h(R(X,Y)Z,W)\\
&&-h(Z,R(X,Y)W),
\eeas
for any $X,Y,Z,W\in \Gamma (TM)$ and $\sigma \in \Gamma (TM^{\perp })$. It
can be verified that 
\begin{equation*}
\langle (\bar{R}(X,Y)h)(Z,W),\sigma \rangle =\langle (\bar{R}(X,Y)A)_{\sigma
}Z,W\rangle .
\end{equation*}

A submanifold $M$ is said to be \emph{pseudo parallel} if 
\begin{align*}
(\bar R(X,Y)h)(Z,W)=f[(X\wedge Y)h](Z,W)
\end{align*}
for any $X,Y,Z,W\in\Gamma(TM)$, 
where $f\in\Omega^0(M)$, is called the \emph{associated function}, and 
\begin{align*}
(X\wedge Y)Z &=\langle Y,Z\rangle X-\langle X,Z\rangle Y, \\
[(X\wedge Y)h](Z,W) &=-h((X\wedge Y)Z,W)-h(Z,(X\wedge Y)W), \\
[(X\wedge Y)A]_\sigma Z &=(X\wedge Y)A_\sigma Z-A_\sigma(X\wedge Y)Z.
\end{align*}
If the associated function $f=0$ then the submanifold $M$ is said to be 
\emph{semi parallel}.

Now, let $\hat{M}$ be a Kaehler manifold with complex structure $J$. 
For any $X\in \Gamma (TM)$ and $\sigma \in \Gamma (TM^{\perp })$, we denote 
the tangential  (resp. normal) part of $JX$ and $J\sigma$ by $\phi X$  and $B\sigma$ (resp. $\omega X$ and $C\sigma$) respectively.
From the parallelism of $J$, we have (cf. \cite[pp. 77]{yano-kon}) 
\begin{align}
(\bar{\nabla}_{X}\phi )Y& =A_{\omega Y}X+Bh(X,Y)
\label{eqn:structure-general1} \\
(\bar{\nabla}_{X}\omega )Y& =-h(X,\phi Y)+Ch(X,Y)
\label{eqn:structure-general2}
\end{align}%
for any $X$, $Y\in \Gamma (TM)$. 

The maximal $J$-invariant subspace $\mathcal{D}_x$ of the tangent space $T_xM
$, $x\in M$ is given by 
\begin{equation*}
\mathcal{D}_x=T_xM\cap JT_xM. 
\end{equation*}
\begin{definition}[\cite{chen4}]
A submanifold $M$ in a Kaehler manifold $\hat M$ is called a \emph{generic submanifold} if the dimension of $\mathcal D_x$ is constant along $M$.
The distribution $\mathcal D : x\rightarrow \mathcal D_x$, $x\in M$ is called the 
\emph{holomorphic distribution (or Levi distribution)} on $M$ and the complex dimension of $\mathcal D$ is called the CR-dimension of $M$.
\end{definition}

\begin{definition}[\cite{bejancu}] \label{def:cr}
A generic submanifold $M$ in a Kaehler manifold $\hat M$ is called  a \emph{CR-submanifold} if 
the orthogonal complementary distribution $\mathcal D^\perp$ of $\mathcal D$ in $TM$ is totally real,
i.e., $J\mathcal D^\perp\subset TM^\perp$.
The real dimension of $\mathcal D^\perp$ is called the CR-codimension of $M$.

If $\mathcal D^\perp=\{0\}$ (resp. $\mathcal D=\{0\}$), the CR-submanifold $M$ is said to be 
\emph{holomorphic} (resp. \emph{totally real}).
A CR-submanifold $M$ is said to be \emph{proper} if it is neither holomorphic nor totally real.
Let $\nu$ be the orthogonal complementary distribution of $J\mathcal D^\perp$ in $TM^\perp$.
Then an \emph{anti-holomorphic} submanifold $M$ is a CR-submanifold with $\nu=\{0\}$,
i.e., $J\mathcal D^\perp=TM^\perp$.
A \emph{real hypersurface} is a proper CR-submanifold of codimension one.
\end{definition}

For a local frame of orthonormal vectors $E_1,E_2,\cdots,E_{2p}$ in $\Gamma(%
\mathcal{D})$, where $p=\dim_{\mathbb{C}}\mathcal{D}$, we define the \emph{$%
\mathcal{D}$-mean curvature vector} $H_{\mathcal{D}}$ by 
\begin{equation*}
H_{\mathcal{D}}=\frac1{2p}\sum^{2p}_{j=1}h(E_j,E_j). 
\end{equation*}

\begin{lemma}[\cite{loo}]\label{lem:CH_D=0}
Let $M$ be a CR-submanifold in a Kaehler manifold $\hat M$. Then 
$\langle(\phi A_\sigma+A_\sigma\phi)X,Y\rangle=0$,
for any $X,Y\in\Gamma(\mathcal D)$ and $\sigma\in\Gamma(\nu)$. Moreover, we have $CH_{\mathcal D}=0$.
\end{lemma}
If $h(\mathcal{D}^\perp,\mathcal{D})=0$, the CR-submanifold $M$ is said to
be \emph{mixed totally geodesic}. $M$ is said to be \emph{mixed foliate} if
it is mixed totally geodesic and $\mathcal{D}$ is integrable.

The following lemma characterizes mixed foliate CR-submanifolds in a Kaehler
manifold. 
\begin{lemma}[\cite{chen3}] \label{lem:mixed-foliate}
A CR-submanifold $M$ in a Kaehler manifold is mixed foliate if and only if 
$Bh(\phi X,Y)=Bh(X,\phi Y)$, for any $X, Y\in\Gamma(\mathcal D)$ and 
$h(\mathcal D^\perp,\mathcal D)=0$.
\end{lemma}

Now suppose the ambient space is an $n$-dimensional complex space form $\hat
M_n(c)$ with constant holomorphic sectional curvature $4c$. The curvature
tensor $\hat R$ of $\hat M_n(c)$ is given by 
\begin{equation*}
\hat R(X,Y)Z=c(X\wedge Y+JX\wedge JY-2\langle JX,Y\rangle J)Z 
\end{equation*}
for any $X,Y,Z\in\Gamma(T\hat M_n(c))$. 
The equations of Gauss, Codazzi and Ricci are then given respectively by 
\begin{align*}
R(X,Y)Z=c(X\wedge Y+\phi X\wedge \phi Y-2\langle\phi X,Y\rangle\phi)Z
+A_{h(Y,Z)}X-A_{h(X,Z)}Y \\
(\bar\nabla_{X}h)(Y,Z)-(\bar\nabla_{Y}h)(X,Z) = c\{\langle \phi
Y,Z\rangle\omega X-\langle \phi X,Z\rangle\omega Y - 2\langle\phi
X,Y\rangle\omega Z\} \\
R^\perp(X,Y)\sigma = c(\omega X\wedge\omega Y-2\langle\phi X,Y\rangle
C)\sigma+h(X,A_\sigma Y)-h(Y,A_\sigma X)
\end{align*}
for any $X$, $Y$, $Z\in\Gamma(TM)$ and $\sigma\in\Gamma(TM^\perp)$.

We now recall the following known result.

\begin{theorem}[\cite{chen3}, \cite{chen-wu}]\label{thm:mixed-foliate}
There does not exist any proper mixed foliate CR-submanifold in $\hat M_n(c)$, $c\neq0$.
\end{theorem}



\section{Mixed-totally geodesic CR-submanifolds in a complex space form}

A real hypersurface $M$ in a Kaehler manifold is said to be \emph{Hopf} if
it is mixed-totally geodesic. A fundamental property of Hopf hypersurfaces
in $\hat M_n(c)$, $c\neq0$, is that the principal curvature $\alpha$
corresponds to the Reeb vector field $\xi$ is constant. Moreover, the other
principal curvatures could be related to $\alpha$ by a nice formula (cf. 
\cite{ryan}). In this section, we show that these properties hold for
mixed-totally geodesic proper CR-submanifolds of maximal CR-dimension.

Suppose $M$ is a real $(2p+1)$-dimensional CR-submanifold in $\hat M_{n}(c)$
of maximal CR-dimension, that is, $\dim_{\mathbb{C}}\mathcal{D}=p$ and $\dim%
\mathcal{D}^\perp=1$. Let $N\in\Gamma(J\mathcal{D}^\perp)$ be a unit vector
field, $\xi=-JN$ and $\eta$ the 1-form dual to $\xi$. Then we have 
\begin{align}
\phi^2X &=-X+\eta(X)\xi  \label{eqn:contact1} \\
\omega X &=\eta(X)N; \quad B\sigma=-\langle \sigma,N\rangle \xi
\label{eqn:contact2}
\end{align}
for any $X\in\Gamma(TM)$ and $\sigma\in\Gamma(TM^\perp)$. It follows from (%
\ref{eqn:structure-general1}) and (\ref{eqn:structure-general2}) that 
\begin{align}
(\nabla_X\phi)Y &=\eta(Y)A_NX-\langle A_NX,Y\rangle\xi
\label{eqn:structure1} \\
\nabla_X\xi &=\phi A_NX; \quad \nabla^\perp_XN=Ch(X,\xi)
\label{eqn:structure2} \\
h(X,\phi Y) &=-\langle \phi A_NX,Y\rangle N-\eta(Y)Ch(X,\xi)+Ch(X,Y) 
\label{eqn:structure3}
\end{align}
for any $X$, $Y\in\Gamma(TM)$ and $\sigma\in\Gamma(TM^\perp)$. 

The equations of Codazzi and Ricci can also be reduced to 
\bea
(\bar\nabla_{X}h)(Y,Z)-(\bar\nabla_{Y}h)(X,Z) &=& c\{\eta(X)\langle \phi
Y,Z\rangle-\eta(Y)\langle \phi X,Z\rangle 
\notag\\&&- 2\eta(Z)\langle\phi X,Y\rangle\}N
\eea
\begin{align}
R^\perp(X,Y)\sigma =-2c\langle\phi X,Y\rangle C\sigma+h(X,A_\sigma
Y)-h(Y,A_\sigma X) 
\end{align}
for any $X$, $Y$, $Z\in\Gamma(TM)$ and $\sigma\in\Gamma(TM^\perp)$. 


\begin{lemma}\label{lem:alpha}
Let $M$ be a mixed-totally geodesic proper CR-submanifold of maximal CR-dimension in $\hat M_{n}(c)$, $c\neq0$,
and let $\alpha=\langle h(\xi,\xi),N\rangle$. Then
\begin{enumerate}
	\item[(a)] $2A_N\phi A_N-\alpha(\phi A_N+A_N\phi)-2c\phi=0$;
	\item[(b)] if $A_NY=\lambda Y$ and $A_N\phi Y=\lambda^*\phi Y$, where $Y\in\Gamma(\mathcal D)$, then\\
	$(2\lambda-\alpha)(2\lambda^*-\alpha)=\alpha^2+4c$;
	\item[(c)]  $\alpha$ is a constant. 
	\end{enumerate}
\end{lemma}

\begin{proof}
By the hypothesis,  
\begin{align}\label{eqn:mixed-geodesic-def}
h(Y,\xi)=\eta(Y)h(\xi,\xi)
\end{align}
for any $Y\in\Gamma(TM)$. 
Differentiating covariantly both sides of (\ref{eqn:mixed-geodesic-def})
in the direction of $X\in\Gamma(TM)$, we get 
\[
(\bar\nabla_Xh)(Y,\xi)+h(\phi A_NX,Y)=\langle\phi A_NX,Y\rangle h(\xi,\xi)+\eta(Y)\nabla^\perp_X h(\xi,\xi).
\]
By applying the Codazzi equation and this equation, we have
\bea
&h(\phi A_NX,Y)-h(X,\phi A_NY)-\langle (\phi A_N+A_N\phi)X,Y\rangle h(\xi,\xi)\notag\\
&-2c\langle\phi X,Y\rangle N=\eta(Y)\nabla^\perp_X h(\xi,\xi)-\eta(X)\nabla^\perp_Y h(\xi,\xi).
\eea
By putting $Y=\xi$ in this equation, we obtain
\begin{align}\label{eqn:del_h:01}
\nabla^\perp_X h(\xi,\xi)=\eta(X)\nabla^\perp_{\xi} h(\xi,\xi)
\end{align}
and 
\bea\label{eqn:del_h:02}
&h(\phi A_NX,Y)-h(X,\phi A_NY)-\langle(\phi A_N+A_N\phi)X,Y\rangle h(\xi,\xi)
\notag\\&=2c\langle\phi X,Y\rangle N.
\eea
By taking inner product of (\ref{eqn:del_h:02}) with $N$, we get
\[ 
2A_N\phi A_N-\alpha(\phi A_N+A_N\phi)-2c\phi=0.
\] 
Statement (b) is directly from this equation.
Next, it follows from (\ref{eqn:structure2}), (\ref{eqn:mixed-geodesic-def}),
and (\ref{eqn:del_h:01}) that 
\[
Y\alpha=Y\langle h(\xi,\xi),N\rangle=g\eta(Y)
\]
for any $Y\in\Gamma(TM)$, where $g=\xi\alpha$, i.e., $d\alpha=g\eta$. Hence 
\[
0=d^2\alpha=dg\wedge\eta+gd\eta.
\]
Since $2d\eta(X,\xi)=\langle(\phi A_N+A_N\phi)X,\xi\rangle=0$ and 
$Xg-(\xi g)\eta(X)=dg\wedge\eta(X,\xi)$, for any $X\in\Gamma(TM)$,  we have $dg=(\xi g)\eta$. Hence we have $gd\eta=0$.
This implies that $g=0$ (for otherwise, 
if $d\eta=0$ then $\mathcal D$ is integrable. It follows that $M$ is mixed foliate but this contradicts
Theorem~\ref{thm:mixed-foliate}). Hence we have $d\alpha=0$ or $\alpha$ is a constant.
\end{proof}



\section{Proof of Theorem~\protect\ref{thm:main}}

Throughout this section, suppose $M$ is a $(2p+q)$-dimensional pseudo
parallel proper CR-submanifold in $\hat M_{n}(c)$, $c\neq0$, where $\dim_{%
\mathbb{C}}\mathcal{D}=p\geq2$ and $\dim_{\mathbb{R}}\mathcal{D}^\perp =q$.

Note that $\mathfrak{S}_{X,Y,Z}((X\wedge Y)h)(Z,W)=0$ and 
\begin{equation*}
\mathfrak{S}_{X,Y,Z}(\bar R(X,Y)h)(Z,W)=\mathfrak{S}_{X,Y,Z}\{R^%
\perp(X,Y)h(Z,W)-h(Z,R(X,Y)W)\} 
\end{equation*}
for any $X,Y,Z,W\in\Gamma(TM)$, where $\mathfrak{S}_{X,Y,Z}$ denotes the
cyclic sum over $X,Y$ and $Z$. By the Gauss and Ricci equations, we obtain
the following equation. 
\begin{align}  \label{eqn:general}
&\langle \omega Y,h(Z,W)\rangle\langle\omega X,\sigma\rangle-\langle \omega
X,h(Z,W)\rangle\langle\omega Y,\sigma\rangle -2\langle\phi X,Y\rangle\langle
Ch(Z,W),\sigma\rangle  \notag \\
&+\langle \omega Z,h(X,W)\rangle\langle\omega Y,\sigma\rangle-\langle \omega
Y,h(X,W)\rangle\langle\omega Z,\sigma\rangle -2\langle\phi Y,Z\rangle\langle
Ch(X,W),\sigma\rangle  \notag \\
&+\langle \omega X,h(Y,W)\rangle\langle\omega Z,\sigma\rangle-\langle \omega
Z,h(Y,W)\rangle\langle\omega X,\sigma\rangle -2\langle\phi Z,X\rangle\langle
Ch(Y,W),\sigma\rangle  \notag \\
&-\langle\phi Y,W\rangle\langle h(Z,\phi X),\sigma\rangle+\langle\phi
X,W\rangle\langle h(Z,\phi Y),\sigma\rangle +2\langle\phi X,Y\rangle\langle
h(Z,\phi W),\sigma\rangle  \notag \\
&-\langle\phi Z,W\rangle\langle h(X,\phi Y),\sigma\rangle+\langle\phi
Y,W\rangle\langle h(X,\phi Z),\sigma\rangle +2\langle\phi Y,Z\rangle\langle
h(X,\phi W),\sigma\rangle  \notag \\
&-\langle\phi X,W\rangle\langle h(Y,\phi Z),\sigma\rangle+\langle\phi
Z,W\rangle\langle h(Y,\phi X),\sigma\rangle +2\langle\phi Z,X\rangle\langle
h(Y,\phi W),\sigma\rangle\notag\\
&=0.
\end{align}
for any $X,Y,Z,W\in\Gamma(TM)$ and $\sigma\in\Gamma(TM^\perp)$. By putting $%
Z\in\Gamma(TM)$, $W\in\Gamma(D^\perp)$, $Y=\phi X$, $X\in\Gamma(\mathcal{D})$
with $||X||=1$ and $X\perp Z,\phi Z$ in (\ref{eqn:general}), we obtain 
\begin{align}  \label{eqn:Ch(D^perp,TM)=0}
Ch(\mathcal{D}^\perp, TM)=0.
\end{align}

Let $\{E_{1},E_{2},\cdots ,E_{2p}\}$ be a local orthonormal frame on $%
\mathcal{D}$. By putting $X=E_{j}$, $Z=\phi E_{j}$ for $j\in \{1,2,\cdots
,2p\}$ in (\ref{eqn:general}), and then summing up these equations, with the
help of (\ref{eqn:Ch(D^perp,TM)=0}), we obtain 
\begin{align}
 (2p-2)Ch(Y,W)-2p\langle \phi Y,W\rangle H_{\mathcal{D}}-h(\phi ^{2}W,\phi Y)&  \nonumber  \label{eqn:110} \\
-2h(\phi ^{2}Y,\phi W)-(2p+1)h(Y,\phi W)& =0
\end{align}%
for any $Y,W\in \Gamma (TM)$. By virtue of (\ref{eqn:Ch(D^perp,TM)=0}),
after putting $Y\in \Gamma (\mathcal{D}^{\perp })$ in the above equation, we
have 
\begin{equation} \label{eqn:mixed-geodesic}
h(\mathcal{D}^{\perp },\mathcal{D})=0.
\end{equation}%
This means that $M$ is mixed-totally geodesic and so (\ref{eqn:110}) reduces
to 
\begin{equation} \label{eqn:130}
(2p-2)Ch(Y,W)-2p\langle \phi Y,W\rangle H_{\mathcal{D}}+h(W,\phi
Y)-(2p-1)h(Y,\phi W)=0
\end{equation}%
for any $Y,W\in \Gamma (TM)$. Next, we put $Y=W$ in the above equation to
get $Ch(Y,Y)-h(Y,\phi Y)=0$, then, combining with the linearity of $C$, $h$
and $\phi $, we obtain 
\begin{equation} \label{eqn:D-umbilic}
2Ch(Y,W)-h(W,\phi Y)-h(Y,\phi W)=0
\end{equation}
for any $Y,W\in \Gamma (TM)$. It follows from this equation and (\ref%
{eqn:130}) that 
\begin{equation}
h(Y,\phi W)=\langle Y,\phi W\rangle H_{\mathcal{D}}+Ch(Y,W)\label{upa}
\end{equation}
for any $Y,W\in \Gamma (TM)$. From (\ref{eqn:general}) and (\ref%
{upa}), we have 
\begin{align*}
& \langle \omega Y,h(Z,W)\rangle \omega X-\langle \omega X,h(Z,W)\rangle
\omega Y+\langle \omega Z,h(X,W)\rangle \omega Y \\
& -\langle \omega Y,h(X,W)\rangle \omega Z+\langle \omega X,h(Y,W)\rangle
\omega Z-\langle \omega Z,h(Y,W)\rangle \omega X=0
\end{align*}%
for any $X,Y,Z,W\in \Gamma (TM)$.

We claim that $q=1$. Suppose the contrary that $q\geq 2$. By putting $Z=W\in
\Gamma (\mathcal{D})$, $Y=BH_{\mathcal{D}}$ and $X\perp BH_{\mathcal{D}}$ a
unit vector field in $\mathcal{D}^{\perp }$ in this equation, with the help
of (\ref{eqn:D-umbilic}), we obtain $BH_{\mathcal{D}}=0$. This, together with
(\ref{eqn:D-umbilic}) imply that $h(\mathcal{D},\mathcal{D})=0$ and hence,
by Lemma~\ref{lem:mixed-foliate} and (\ref{eqn:mixed-geodesic}), $M$ is
mixed foliate. This contradicts Theorem~\ref{thm:mixed-foliate}.
Accordingly, $q=1$. 

Let $N\in\Gamma(J\mathcal{D}^\perp)$ be a unit vector field normal to $M$,
and $(\phi, \eta,\xi)$ the almost contact structure on $M$ as defined in
Section 3. 
It follows from Lemma~\ref{lem:CH_D=0} and equations (\ref{eqn:contact1}), (%
\ref{eqn:contact2}), (\ref{eqn:Ch(D^perp,TM)=0}) and (\ref%
{eqn:mixed-geodesic}) that 
\begin{align}
H_{\mathcal{D}}=&\lambda N,  \label{eqn:HD}\quad \\
h(X,\xi)=&\eta(X)h(\xi,\xi)=\alpha\eta(X)N \nonumber
\end{align}
for any $X\in\Gamma(TM)$, where $\lambda=\langle H_{\mathcal{D}},N\rangle$
and $\alpha=\langle h(\xi,\xi),N\rangle$. By using (\ref{eqn:D-umbilic}) and
the above two equations, we obtain 
\begin{align}  \label{eqn:h}
h(X,Y)
&=h(X,-\phi^2 Y+\eta(Y)\xi)	\nonumber\\
&=\{\lambda\langle X,Y\rangle+b\eta(X)\eta(Y)\}N-Ch(X,\phi Y)
\end{align}
for any $X,Y\in\Gamma(TM)$, where $b=\alpha-\lambda$. From Lemma~\ref{lem:alpha} and (\ref{eqn:h}), we obtain 
\begin{align}  \label{eqn:lambda}
\lambda^2-\alpha\lambda-c=0
\end{align}
and so $\lambda$ is a non-zero constant. Further, for any unit vector $Y\in \mathcal D$, we have 
\begin{align*}
0=\langle(\bar R(\xi,Y)h)(Y,\xi),N\rangle\rangle-f\langle((\xi\wedge
Y)h)(Y,\xi),N\rangle =(\alpha-\lambda)(f-\alpha\lambda-c)
\end{align*}
Hence, $f=\lambda^2$ is a positive constant.

We consider two cases: $Ch=0$ and $Ch\neq0$.

\medskip \textbf{Case 1.} $Ch=0$.

By the hypothesis, (\ref{eqn:structure2}) and the fact that $\lambda\neq0$,
the first normal space $\mathcal{N}^1_x=\mathbb{R}N_x$, $x\in M$, and $%
\mathcal{N}^1$ is a parallel normal subbundle of $TM^\perp$. Since $\nu$ is $%
J$-invariant, by Codimension Reduction Theorems (cf. \cite{djoric}, \cite%
{kawamoto}), $M$ is contained in a totally geodesic holomorphic submanifold $%
\hat M_{p+1}(c)$ as a real hypersurface.

Now, let $\nabla^{\prime }$, $A^{\prime }$, \emph{etc} denote the
Levi-Civita connection on $M$ induced by the Levi-Civita connection of $\hat
M_{p+1}(c)$, the shape operator, \emph{etc}, respectively. Since $\hat
M_{p+1}(c)$ is totally geodesic in $\hat M_{n}(c)$, we can see that $%
\nabla^{\prime }_XY=\nabla_XY$, $A^{\prime }=A_N$ and $N^{\prime }=N$.
Further, as $\nabla^\perp N=0$, we have $R^\perp(X,Y)N=0$ and so $R^{\prime
}(X,Y)A=(\bar R(X,Y)A)_N$, for any $X,Y$ tangent to $M$. Then $M$ is a
pseudo parallel real hypersurface in $\hat M_{p+1}(c)$ and by Theorem~\ref%
{thm:pseudo-parallel-hypersurface}, we obtain List (a) and (b-i) in Theorem~
\ref{thm:main}.

\medskip \textbf{Case 2.} $Ch\neq0$.

Suppose $Ch\neq0$ at a point $x\in M$. There is a number $a\neq0$, $%
\sigma\in\nu_x$ and a unit vector $Y\in\mathcal{D}_x$ such that $A_\sigma
Y=aY$. From Lemma~\ref{lem:CH_D=0}, we have $A_\sigma\phi Y=-a\phi Y$. Then
from $\langle (\bar R(\phi Y,Y)h)(Y,\phi Y),\sigma\rangle=f\langle((\phi
Y\wedge Y)h)(Y,\phi Y),\sigma\rangle$, we obtain 
\begin{equation*}
a\{3c-2\langle h(Y,\phi Y),h(Y,\phi Y)\rangle+\langle h(Y,Y),h(\phi Y,\phi
Y)\rangle\}=af. 
\end{equation*}
On the other hand, from (\ref{eqn:h}), we have 
\begin{align*}
&\langle h(Y,\phi Y),h(Y,\phi Y)\rangle=\langle Ch(Y,Y),Ch(Y,Y)\rangle \\
&\langle h(Y,Y),h(\phi Y,\phi Y)\rangle=\lambda^2- \langle
Ch(Y,Y),Ch(Y,Y)\rangle.
\end{align*}
Since $a\neq0$ and $f=\lambda^2$, these equations give $c=\langle
Ch(Y,Y),Ch(Y,Y)\rangle$. Hence, we conclude that $c>0$ (without loss of
generality, we assume $c=1$) and $||Ch||>0$ on the whole of $M$.

Fixed $r>0$ and let $BM$ be the unit normal bundle over $M$. The focal map $%
\Phi_r$ is given by 
\begin{equation*}
BM\ni \sigma\overset{\Phi_r}{\longrightarrow}\exp(r\sigma)\in \mathbb{C}P_n 
\end{equation*}
where $\exp$ is the exponential map on $\mathbb{C}P_n$. For each $x\in M$
and unit vector $\sigma\in T_xM^\perp$, denote by $\gamma_{\sigma}(s)$ the
normalized geodesic in $\mathbb{C}P_n$ passes through $x\in M$ at $s=0$ with
velocity $\sigma$. Let $\mathcal{Y}_X$ be the $M$-Jacobi field along $%
\gamma_{\sigma}$ with initial values $\mathcal{Y}_X(0)=X\in T_xM$ and $\dot{%
\mathcal{Y}}_X(0)=-A_{\sigma}X$. Then (cf. \cite[pp.225]{berndt}) 
\begin{equation*}
d\Phi_r(\sigma)X=\mathcal{Y}_X(r). 
\end{equation*}
In view of (\ref{eqn:h}), $A_N$ has two distinct constant eigenvalues $\alpha
$ and $\lambda$ with eigenspaces $\mathbb{R}\xi$ and $\mathcal{D}_x$
respectively at each $x\in M$. 
We put $\alpha=2\cot 2r$, $0<r<\pi/2$. Then $\lambda=\cot r$
or $\lambda=-\cot(\frac\pi2-r)$ by (\ref{eqn:lambda}).

\medskip \textbf{Subcase 2-a.} $\lambda=\cot r$.

Since $\lambda$ is a nonzero constant, by (\ref{eqn:HD}), $N=\lambda^{-1}H_{\mathcal D}$ is globally
defined on $M$. We may immerse $M$ in $BM$ as a submanifold in a natural
way: $x\mapsto N_x$, $x\in M$.

We claim that $\Phi_r(M)$ is a singleton for a suitable choice of $r$. This
can be done by showing that $d\Phi_r(N_x)T_xM=\{0\}$, for each $x\in M$. We
first note that at each $z\in \mathbb{C}P_n$, the Jacobi operator $\hat
R_\sigma:=\hat R(\cdot,\sigma)\sigma$, $\sigma\in T_z\mathbb{C}P_n$, has
eigenvalues $0$, $4$ and $1$ with eigenspaces $\mathbb{R}\sigma$, $\mathbb{R}%
J\sigma$ and $(\mathbb{R}\sigma\oplus\mathbb{R}J\sigma)^\perp$ respectively,
To compute $d\Phi_r(N_x)X$, $X\in T_xM$, we select the Jacobi field 
\begin{align*}
\mathcal{Y}_X(t)=\left\{%
\begin{array}{rl}
\left(\cos 2t-\frac{\alpha}2\sin 2t\right)\mathcal{E}_X(t), & X=\xi \\ 
(\cos t-\lambda\sin t)\mathcal{E}_X(t), & X\in\mathcal{D}_x%
\end{array}%
\right.
\end{align*}
where $\mathcal{E}_X$ is the parallel vector field along $\gamma_{N_x}$ with 
$\mathcal{E}_X(0)=X$. Then we have $d\Phi_r(N_x)X=\mathcal{Y}_X(r)=0$ and
conclude that $\Phi_r(M)=\{z_0\}$.

\medskip \textbf{Subcase 2-b.} $\lambda=-\cot(\frac\pi2-r)$.

Note that $\cot 2r=-\cot2(\frac\pi2-r)$. By selecting the Jacobi field 
\begin{align*}
\mathcal{Y}_X(t)=\left\{%
\begin{array}{rl}
\left(\cos 2t+\frac{\alpha}2\sin 2t\right)\mathcal{E}_X(t), & X=\xi \\ 
(\cos t+\lambda\sin t)\mathcal{E}_X(t), & X\in\mathcal{D}_x%
\end{array}%
\right.
\end{align*}
we can see that $d\Phi_{\pi/2-r}(-N_x)X=0$, for $X\in T_xM$ and hence $%
\Phi_{\pi/2-r}(M)=\{z_0\}$.

\medskip We have shown that $\Phi_r(M)=\{z_0\}$ for some $r\in]0,\pi/2[$ in
both cases. By checking the Jacobi fields of $\mathbb{C}P_n$ (cf. \cite[%
pp.149]{gallot}), there is no conjugate point for $z_0$ along any geodesic
in $\mathbb{C}P_n$ of length $r\in]0,\pi/2[$ starting at $z_0$, we conclude
that $M$ lies in a geodesic hypersphere $M^{\prime }$ around $z_0$ in $%
\mathbb{C}P_n$ with almost contact structure $(\phi^{\prime },\eta^{\prime
},\xi^{\prime })$, where $\xi^{\prime }=-JN^{\prime }$, $\eta^{\prime }$ the 
$1$-form dual to $\xi^{\prime }$, $\phi^{\prime }=J_{|TM^{\prime
}}-\eta^{\prime }\otimes N^{\prime }$ and $N^{\prime }$ a unit vector field
normal to $M^{\prime }$. By the construction of $M^{\prime }$, we have $%
N=N^{\prime }$, $\xi=\xi^{\prime }$ and $\phi=\phi^{\prime }$ on $M$. It
follows that $\phi^{\prime }TM\subset TM$ and so $M$ is an invariant
submanifold of $M^{\prime }$ (cf. \cite{yano-kon}). Hence we obtain List
(b-ii) in Theorem~\ref{thm:main}.

\section{Recurrent CR-submanifolds in a non-flat complex space form}
In this section, wel show that there are no proper recurrent CR-submanifolds in $\hat M_n(c)$, $n\neq0$.
We first discuss the ideas of recurrence and semi-parallelism in a general setting.

Let $M$ be a Riemannian manifold and $\mathcal E_j$ a Riemannian vector bundle over $M$ with linear connection $\nabla^j$, $j\in\{1,2\}$.
It is known that $\mathcal E^*_1\otimes \mathcal E_2$ is isomorphic to the vector bundle $Hom(\mathcal E_1,\mathcal E_2)$, consisting of homomorphisms from $\mathcal E_1$ into $\mathcal E_2$. 
We denote by the same $\langle,\rangle$ the fiber metrics on $\mathcal E_1$ and $\mathcal E_2$ as well as that induced on $Hom(\mathcal E_1,\mathcal E_2)$.
The connections  $\nabla^1$ and $\nabla^2$ induce on $Hom(\mathcal E_1,\mathcal E_2)$ a connection $\bar 
\nabla$, given by
\[
(\bar\nabla_XF)V=(\bar\nabla F)(V;X)=\nabla^2_XFV-F\nabla^1_XV
\]
for any $X\in\Gamma(TM)$, $V\in\Gamma(\mathcal E_1)$ and $F\in\Gamma(Hom(\mathcal E_1,\mathcal E_2))$.

A section $F$ in $Hom(\mathcal E_1,\mathcal E_2)$ is said to be \emph{recurrent} if there exists $\tau\in\Omega^1(M)$ such that 
$\bar\nabla F=F\otimes \tau$. 
We may regard parallelism as a special case of recurrence, that is, the case $\tau=0$.
Let $\bar R$, $R^1$ and $R^2$ be the curvature tensor corresponding to $\bar\nabla$, $\nabla^1$ and $\nabla^2$ respectively.
Then we have
\[
(\bar R\cdot F)(V;X,Y)=(\bar R(X,Y)F)V=R^2(X,Y)FV-FR^1(X,Y)V
\] 
for any $X,Y\in\Gamma(TM)$, $V\in\Gamma(\mathcal E_1)$ and $F\in\Gamma(Hom(\mathcal E_1,\mathcal E_2))$.

We begin with the following result.

\begin{lemma}\label{lem:recurrent}
Let $M$ be a connected Riemannian manifold, $\mathcal E_j$ a Riemannian vector bundle over $M$, $j\in\{1,2\}$ and 
$F\in\Gamma(Hom(\mathcal E_1,\mathcal E_2))$.
If $F$ is recurrent then $F$ is semi-parallel. 
\end{lemma}
\begin{proof}
Suppose $F$ is recurrent, that is, $\bar\nabla F=F\otimes\tau$, for some $\tau\in\Omega^1(M)$.
It is trivial if $F=0$. Suppose that $\mu:=||F||\neq0$ on an open set $ U\subset M$.
Then the line bundle $\mathbb R\otimes F\rightarrow U$, spanned by $F$, is a parallel subbundle of 
$Hom(\mathcal E_1,\mathcal E_2)_{|U}$.
Consider the unit section $E:=\mu^{-1}F$ of $\mathbb R\otimes F$. Then
\begin{align*}
\bar\nabla E=\mu^{-1}\bar\nabla F+F\otimes d(\mu^{-1})
=F\otimes(\mu^{-1}\tau+d(\mu^{-1}))=E\otimes(\tau-\mu^{-1}d\mu).
\end{align*}
Hence, $E$ is also recurrent and $\bar\nabla E=E\otimes\lambda$, where $\lambda=\tau-\mu^{-1}d\mu\in\Omega^1(U)$. It follows that
\[
0=d\langle E,E\rangle=2\langle\bar\nabla E,E\rangle=2\langle E,E\rangle \lambda=2\lambda.
\]
Hence $E$ is a flat section.
This implies that $\mathbb R\otimes F$ is a flat bundle. Hence, $\bar R\cdot F=0$
 on $ U$. By a standard topological argument, we conclude that $\bar R\cdot F=0$ on $M$.
\end{proof}

Geometrically, Lemma~\ref{lem:recurrent} tells us that the line subbundle of $(Hom(\mathcal E_1,\mathcal E_2), \bar\nabla)$, spanned by a nonvanishing recurrent section is a flat bundle.

A submanifold $M$ of a Riemannian manifold $\hat M$ is said to be \emph{recurrent} if its second fundamental form $h$ is recurrent. 
Since every $T_xM^\perp$-valued bilinear map on $T_xM$ naturally induces a homomorphism from $T_xM\otimes T_xM$ to $T_xM^\perp$,
$x\in M$,
we may identify $h$ as a section of $Hom(TM\otimes TM,TM^\perp)$. Accordingly, 
the following result can be obtained immediately from Corollary~\ref{cor:semi-parallel} and Lemma~\ref{lem:recurrent}.

\begin{corollary}\label{cor:recurrent}
There does not exist any proper recurrent CR-submanifold $M$ in $\hat M_n(c)$, $c\neq0$, with 
$\dim_{\mathbb C}\mathcal D\geq2$.
\end{corollary}
\begin{remark}
The above corollary generalizes the non-existence of recurrent real hypersurfaces in a non-flat complex space form
(cf. \cite{hamada}, \cite{lyu-suh}).
\end{remark}

\section*{Acknowledgement}
The authors are thankful to the referee for several valuable comments and suggestions towards the improvement of the present article.



\end{document}